\documentclass[11pt,english]{article}

\usepackage[margin = 2.4cm]{geometry}

\usepackage{amsthm}
\usepackage{amsmath}
\usepackage{amssymb}
\usepackage{mathtools}
\usepackage{graphicx}
\usepackage[hidelinks, bookmarks = false]{hyperref}
\usepackage{enumitem}

\usepackage{caption}
\usepackage[square,sort,comma,numbers]{natbib} 
\usepackage{setspace}
\usepackage{cleveref}
\theoremstyle{plain}

\newtheorem*{thm*}{Theorem}
\newtheorem{thm}{Theorem}[section]
\crefname{thm}{Theorem}{Theorems}
\Crefname{thm}{Theorem}{Theorems}

\newtheorem*{lem*}{Lemma}
\newtheorem{lem}[thm]{Lemma}
\crefname{lem}{Lemma}{Lemmas}
\Crefname{lem}{Lemma}{Lemmas}

\newtheorem*{claim*}{Claim}
\newtheorem{claim}[thm]{Claim}
\crefname{claim}{Claim}{Claims}
\Crefname{claim}{Claim}{Claims}

\crefname{prop}{Proposition}{Propositions}
\Crefname{prop}{Proposition}{Propositions}

\newtheorem{cor}[thm]{Corollary}
\crefname{cor}{Corollary}{Corollaries}
\Crefname{cor}{Corollary}{Corollaries}

\crefname{conj}{Conjecture}{Conjectures}
\Crefname{conj}{Conjecture}{Conjectures}

\crefname{qn}{Question}{Questions}
\Crefname{qn}{Question}{Questions}

\crefname{obs}{Observation}{Observations}
\Crefname{obs}{Observation}{Observations}

\crefname{ex}{Example}{Examples}
\Crefname{ex}{Example}{Examples}

\newtheorem{rem}[thm]{Remark}
\crefname{rem}{Remark}{Remarks}
\Crefname{rem}{Remark}{Remarks}

\theoremstyle{definition}

\crefname{prob}{Problem}{Problems}
\Crefname{prob}{Problem}{Problems}

\newtheorem{defn}[thm]{Definition}
\crefname{defn}{Definition}{Definitions}
\Crefname{defn}{Definition}{Definitions}

\theoremstyle{remark}

\usepackage{xpatch}
\xpatchcmd{\proof}{\itshape}{\normalfont\proofnamefont}{}{}
\newcommand{\proofnamefont}{}
\renewcommand{\proofnamefont}{\bfseries}

\usepackage{framed}

\newcommand{\remove}[1]{}

\newcommand{\ceil}[1]{
    \left\lceil #1 \right\rceil
}
\newcommand{\floor}[1]{
    \left\lfloor #1 \right\rfloor
}
\newcommand{\comp}[1]{\overline{#1}}

\newcommand{\eps}{\varepsilon}
\newcommand{\G}{\mathcal{G}}

\usepackage{bm}

\DeclareMathOperator{\cm}{CM}
\newcommand{\even}[1]{\langle\langle #1 \rangle\rangle}
\newcommand{\odd}[1]{\langle #1 \rangle}

\renewcommand{\Pr}{\mathbb{P}}
\newcommand{\Ex}{\mathbb{E}}

\title{An improvement on {\L}uczak's connected matchings method}

\author{
	    Shoham Letzter\thanks{
		Department of Mathematics, 
		University College London, 
		Gower Street, London WC1E~6BT, UK. 
		Email: \texttt{s.letzter}@\texttt{ucl.ac.uk}. 
		Research supported by the Royal Society.
    }
}

\begin{document}

\date{}
\maketitle

\begin{abstract}

	\setlength{\parskip}{\medskipamount}
    \setlength{\parindent}{0pt}
    \noindent

	A \emph{connected matching} in a graph $G$ is a matching contained in a connected component of $G$.
	A well-known method due to {\L}uczak reduces problems about monochromatic paths and cycles in complete graphs to problems about monochromatic connected matchings in almost complete graphs.
	We show that these can be further reduced to problems about monochromatic connected matchings in complete graphs. 

	We illustrate the potential of this new reduction by showing how it can be used to determine the $3$-colour Ramsey number of long paths, using a simpler argument than the original one by Gy\'arf\'as, Ruszink\'o, S\'ark\"ozy, and Szemer\'edi (2007).

\end{abstract}

\section{Introduction} \label{sec:intro}

	The \emph{$k$-colour Ramsey number} of a graph $H$, denoted $r_k(H)$, is the minimum $N$ such that every $k$-edge-colouring of $K_N$ contains a monochromatic copy of $H$. 
	The study of Ramsey-type problems for paths and cycles was initiated by Gerencs\'er and Gy\'arf\'as \cite{gerencser-gyarfas} in an early paper (1967) in which they determined the $2$-colour Ramsey number of a path, showing that $r_2(P_n) = \floor{\frac{3n-2}{2}}$ for $n \ge 2$. Quite a few results in the area, mostly about two colours, were proved in the following few years (see, e.g., \cite{faudree-schelp,faudree-schelp2,rosta,gyarfas-lehel}).
	A while later, in 1999, {\L}uczak \cite{luczak} determined, asymptotically, the $3$-colour Ramsey number of odds cycles. Since then, the $3$-colour Ramsey number of $P_n$ and $C_n$ has been determined precisely for every large $n$ (see \cite{figaj-luczak,kohayakawa-simonovits-skokan,gyarfas-et-al,benevides-skokan}). For $k \ge 4$, the $k$-colour Ramsey number of a path is still unknown; the best known bounds to date are $(k-1)n + O(1) \le r_k(P_n) \le (k-1/2)n + o(n)$. For the lower bound, see Yongqi, Yuansheng, Feng and Bingxi \cite{yongqi-et-al}; an alternative construction can be obtained from an affine plane of order $k-1$ whenever such a plane exists. For the upper bound see Knierim and Su \cite{knierim-su}.
	The same bounds also hold for $r_k(C_n)$ whenever $n$ is even, and are the state of the art in this case too. Remarkably, for long odd cycles the $k$-colour Ramsey number is known precisely: it is $r_k(C_n) = 2^{k-1}(n-1) + 1$ for sufficiently large $n$ (Jenssen and Skokan \cite{jenssen-skokan}). Interestingly, this bound does not holds for all $k$ and $n$ (Day and Johnson \cite{day-johnson}).

	A \emph{connected matching} is a matching contained in a connected component. A \emph{monochromatic connected matching} is a matching contained in a monochromatic component; similarly, an $\ell$-coloured connected matching is a matching contained in a component of colour $\ell$. 
	A key method in the study of Ramsey-type problems about paths and cycles, introduced by {\L}uczak \cite{luczak}, is the use of connected matchings in conjunction with Szemer\'edi's regularity lemma \cite{szemeredi}. {\L}uczak noticed that by applying the regularity lemma, problems about monochromatic paths and cycles in complete graphs can be reduced to problems about monochromatic connected matchings in almost complete graphs. To illustrate this, here is a special case of a lemma by Figaj and {\L}uczak \cite{figaj-luczak} which formalises this reduction.

	\begin{lem}[Special case of Lemma 3 in \cite{figaj-luczak}] \label{lem:figaj-luczak}
		Let $\alpha > 0$ and let $k$ be an integer.
		Suppose that for every $\eps > 0$ there exists $\delta > 0$ such that for sufficiently large $n$ the following holds: every $k$-colouring of every graph $G$ on at least $(1+\eps)\alpha n$ vertices and with density at least $1 - \delta$ contains a monochromatic connected matching on at least $n$ vertices. Then $r_k(C_n) \le (\alpha+o(1))n$ for every \emph{even} n (and so $r_k(P_n) \le (\alpha + o(1))n$).
	\end{lem}

	The connected matchings method, illustrated in \Cref{lem:figaj-luczak}, is clearly very handy; connected matchings are more convenient to work with than paths, and there are several structural results at one's disposal when studying matchings, such as Hall and K\H{o}nig's theorems (see, e.g., \cite{lovasz-plummer}) for bipartite graphs and Tutte's theorem \cite{tutte} and the Gallai-Edmonds decomposition theorem \cite{gallai,edmonds} (see \Cref{thm:ge} below) for general graphs.
	However, the need to switch from complete graphs to almost complete graphs is a drawback. At the very least, it is a nuisance, making proofs more technical and less readable, and in some cases it can be a genuine obstacle.

	Our main aim in this paper is to provide a further reduction, replacing almost complete graphs as in \Cref{lem:figaj-luczak} by complete graphs. To illustrate this, we obtain the following strengthening of \Cref{lem:figaj-luczak} as a corollary of the more general \Cref{thm:main-general} stated below. 

	\begin{cor} \label{cor:main-complete-symmetric}
		Let $\alpha > 0$ and let $k$ be an integer.
		Suppose that for every $\eps > 0$ and every sufficiently large $n$, every $k$-colouring of $K_{\ceil{(1 + \eps)\alpha n}}$ contains a monochromatic connected matching on at least $n$ vertices. Then $r_k(C_n) \le (\alpha + o(1))n$ for every even $n$.
	\end{cor}

	\subsection{$2$-matchings}

		A \emph{$2$-matching} (sometimes known as a \emph{basic $2$-matching}) is a collection of vertex-disjoint edges and odd cycles; the \emph{order} of a $2$-matching is the number of vertices it covers.\footnote{The notion of a $2$-matching is related to that of a fractional matching. A \emph{fractional matching} in a graph $G$ is a function $\omega : E(G) \to [0,1]$ satisfying $\sum_{u \in N(v)} \omega(uv) \le 1$ for every vertex $v$. The \emph{size} of $\omega$ is $\sum_{e \in E(G)} \omega(e)$, and its \emph{order} is twice its size. A simple exercise shows that in every graph $G$, the maximum order of a $2$-matching equals the maximum order of a fractional matching. It thus often suffices to focus on $2$-matchings instead of fractional matchings.}
		A \emph{connected $2$-matching} is a $2$-matching contained in a connected component, and a \emph{monochromatic connected $2$-matching} in an edge-coloured graph is a $2$-matching contained in a monochromatic component. It is not hard to see that the proof of \Cref{lem:figaj-luczak} can be modified to prove a variant of \Cref{lem:figaj-luczak} where in the assumptions instead of matchings we have $2$-matchings. This is a handy improvement, due to the existence of simpler structural results for $2$-matchings (see Pulleyblank \cite{pulleyblank} or \Cref{thm:pulleyblank} below).
		Indeed, a few recent papers (see \cite{debiasio-lo,gyarfas-sarkozy-matchings-few-colours,korandi-lang-letzter-pokrovskiy}) made use of variants of {\L}uczak's method for $2$-matchings. Another consequence of our main result, \Cref{thm:main-general} stated below, is the following variant of \Cref{cor:main-complete-symmetric} for $2$-matchings.

		\begin{cor} \label{cor:main-complete-symmetric-frac}
			Let $\alpha > 0$ and let $k$ be an integer.
			Suppose that for every $\eps > 0$ and every sufficiently large $n$, every $k$-colouring of $K_{\ceil{(1 + \eps)\alpha n}}$ contains a monochromatic connected $2$-matching of order at least $n$. Then $r_k(C_n) \le (\alpha + o(1))n$ for every even $n$.
		\end{cor}

	\subsection{Cycles of different lengths}

		The original statement of Lemma 3 in \cite{figaj-luczak} allows for both odd and even cycles and for different paths or cycle lengths for different colours\footnote{However, the lengths all need to be of the same order of magnitude.}. 
		\Cref{cor:main-complete-symmetric-frac} above can be generalised similarly. Following \cite{figaj-luczak}, given a real number $x$ define $\even{x}$ to be the largest even number not larger than $x$ and $\odd{x}$ to be the largest odd number not larger than $x$. Previously stated results show that monochromatic even cycles correspond to connected matchings of the same colour; the next corollary shows that monochromatic odd cycles correspond to connected matchings that are contained in a non-bipartite monochromatic component of the same colour.

		\begin{cor} \label{cor:main-complete-asymmetric}
			Let $k_0 \le k$ be integers and let $\alpha_1, \ldots, \alpha_k > 0$. Suppose that for every $\eps > 0$ and sufficiently large $n$,
			every $k$-colouring of $K_{\ceil{(1+\eps)n}}$ yields an $\ell$-coloured connected $2$-matching $M$ on at least $\alpha_{\ell} n$ vertices for some $\ell \in [k]$, where $M$ is contained in a non-bipartite $\ell$-coloured component if $\ell \in [k_0+1, k]$.
			Then for every $\eps > 0$ and large $n$, every $k$-colouring of $K_{\ceil{(1+\eps)n}}$ contains an $\ell$-coloured cycle of even length $\even{\alpha_{\ell}n}$ if $\ell \in [k_0]$, or odd length $\odd{\alpha_{\ell}n}$ if $\ell \in [k_0+1, k]$.
		\end{cor}

	\subsection{Varying the ground graph} 
		One can use a variant of \Cref{lem:figaj-luczak} to study \emph{bipartite path-Ramsey numbers}, where the ground graph is a complete bipartite graph, or, more generally, \emph{multipartite path-Ramsey numbers}; this generalisation is not mentioned in \cite{figaj-luczak}, but it can be proved in a very similar way. For some examples using this variant, see \cite{balogh-et-al,gyarfas-et-al-2007,bucic-letzter-sudakov-three,bucic-letzter-sudakov-multicolour}.

		We now state a generalisation of \Cref{cor:main-complete-asymmetric}, where the ground graph can vary and where we obtain a range of cycle lengths in one of the colours. Before doing so, we give a definition. 
		Let $F$ be a graph on vertex set $[s]$, which may have loops, and let $m_1, \ldots, m_s > 0$. Define $F(m_1, \ldots, m_s)$ to be the blow-up of $F$ where vertex $i$ is replaced by a set of $\ceil{m_i}$ vertices, for $i \in [s]$. More precisely, the vertex set of $F(m_1, \ldots, m_s)$ is $V_1 \cup \ldots \cup V_s$, where the sets $V_i$ are pairwise disjoint, $|V_i| = \ceil{m_i}$, and the edges of $F(m_1, \ldots, m_s)$ are the pairs $xy$ (with $x \neq y$, i.e.\ loops are not allowed) such that $x \in V_i$ and $y \in V_j$ for some $ij \in E(F)$. In particular, $V_i$ is a clique if $F$ has a loop at $i$, and is an independent set otherwise.

		\begin{thm} \label{thm:main-general} 
			Let $s$ and $k_0 \le k$ be integers, let $\alpha_1, \ldots, \alpha_k, \beta_1, \ldots, \beta_s > 0$, and let $F$ be a graph on vertex set $[s]$, which may have loops.

			Suppose that for every $\eps > 0$ and large $n$, every $k$-colouring of the blow-up $F((\beta_1 + \eps)n, \ldots, (\beta_s + \eps)n))$ of $F$ yields an $\ell$-coloured connected $2$-matching $M$ on at least $\alpha_{\ell} n$ vertices for some $\ell \in [k]$, where $M$ is contained in a non-bipartite $\ell$-coloured component if $\ell \in [k_0+1, k]$.

			Then for every $\eps > 0$ there exists $T$ such that for sufficiently large $n$, every $k$-colouring of $F((\beta_1 + \eps)n, \ldots, (\beta_s + \eps)n)$ satisfies the following: either for some $\ell \in [k_0]$ there is an $\ell$-coloured $C_t$ for every even $t \in [T, \alpha_{\ell}n]$, or for some $\ell \in [k_0 + 1, k]$ there is an $\ell$-coloured $C_t$ for every integer $t \in [T, \alpha_{\ell} n]$.
		\end{thm}

	\subsection{Overview of the proof}

		Our proof of \Cref{thm:main-general} is based on \Cref{thm:matching-general}, stated in \Cref{sec:matchings}. The latter theorem asserts that given a $k$-coloured almost complete graph $G$ with no monochromatic connected matchings of size $m$, there is a slightly smaller $k$-coloured complete graph $G'$ that avoids monochromatic connected matchings of size $m$ (in fact, it contains a similar but more general result for blow-ups of a fixed graph with different bounds on the largest $\ell$-coloured connected matching for different colours). With \Cref{thm:matching-general} in hand, \Cref{thm:main-general} can be easily derived from {\L}uczak's connected matching method (see \Cref{sec:matchings} for a proof of this derivation).

		To prove \Cref{thm:main-general}, we take $G_1$ to be a largest $k$-multicoloured graph (so edges can have several colours) on $V(G)$ that contains $G$ and that satisfies the same requirement on the size of the maximum monochromatic connected matching. It is easy to see that $G_1$ has relatively few monochromatic components (because two small components of the same colour in $G$ can be joined in $G_1$). A consequence of the Gallai-Edmonds theorem, stated in \Cref{sec:prelims}, implies that every monochromatic component in $G_1$ is a complete blow-up of a star. These two properties can be leveraged to show that any matching in the complement $\comp{G}_1$ of $G_1$ is small. One can thus obtain a suitable $G'$ by removing the vertices of a maximum matching in $\comp{G}_1$ from $G_1$. 

	\subsection{Background}

		In \cite{bucic-letzter-sudakov-three}, together with Buci\'c and Sudakov, we determined, asymptotically, the \emph{bipartite $3$-colour Ramsey number} of paths and even cycles. By a variant of \Cref{lem:figaj-luczak}, to do so it suffices to determine, asymptotically, the size of the largest monochromatic connected matching guaranteed to exist in a $3$-coloured balanced almost complete bipartite graph. As a first step, we determined the size of the largest monochromatic connected matching in a $3$-colouring of $K_{n,n}$, for every $n$. It is often the case that proofs about monochromatic connected matchings in complete graphs (or complete bipartite graphs in this case) can be adapted to almost complete graphs. However, due to the inductive nature of our proof, we were unable to find such an adaptation. Instead, we proved (implicitly; see Theorem 12 in \cite{bucic-letzter-sudakov-three}) a version of \Cref{thm:matching-general}, which reduces the $3$-colour Ramsey question about connected matchings in almost complete bipartite graphs to complete bipartite graphs. 
		Our proof of \Cref{thm:matching-general} is similar, with a main difference being the use of the Gallai-Edmonds decomposition (see \Cref{sec:prelims}) which replaces our use of K\"onig's theorem in \cite{bucic-letzter-sudakov-three}.  
	
	\subsection{Stability}

		We note that our method is also handy when proving stability results about monochromatic connected matchings, allowing one to prove such results for complete graphs instead of almost complete graphs. This is likely to be helpful when determining Ramsey numbers of paths and cycles precisely. See \Cref{subsec:stability} for more details.

	\subsection{Structure of the paper}
	
		In the next section, \Cref{sec:matchings}, we state \Cref{thm:matching-general} which, as mentioned above, reduces Ramsey-type problems about connected matchings in almost complete graphs to complete graphs. We then show how to deduce \Cref{thm:main-general} from \Cref{thm:matching-general} (note that \Cref{cor:main-complete-symmetric,cor:main-complete-symmetric-frac,cor:main-complete-asymmetric} follow immediately from \Cref{thm:main-general}). In \Cref{sec:prelims} we state the Gallai-Edmonds decomposition Theorem (\Cref{thm:ge}) and some consequences of it. We then prove first a simplified version of \Cref{thm:matching-general} and then the general version in \Cref{sec:proofs}. In \Cref{sec:three-colours} we show how our method can be applied to provide a simple proof of the asymptotically tight bound $r_3(P_n) \le (2 + o(1))n$ on the $3$-colour Ramsey number of paths (as mentioned previously, $r_3(P_n)$ is known precisely for large $n$; see \cite{gyarfas-et-al}). 

\section{From almost complete graphs to complete graphs} \label{sec:matchings}

	In this section we state \Cref{thm:matching-general} which allows us to reduce Ramsey-type problems about connected matchings in almost complete graphs to complete graphs, and more generally from almost blow-ups of a fixed graph $F$ to blow-ups of $F$. We then show how to deduce \Cref{thm:main-general} from \Cref{thm:matching-general}, using a variant of \Cref{lem:figaj-luczak}.
	Before stating the lengthy \Cref{thm:matching-general}, we state the following special case of it.

	\begin{thm} \label{thm:matching-complete}
		Let $k$ be an integer, let $\beta \ge \alpha > 0$, and let $\eps > 0$ be sufficiently small. Set $\delta = \frac{\eps}{2} \cdot \big(\frac{\alpha}{16\beta}\big)^{2k}$, and let $N \le \beta n$. 

		Suppose that $G$ is a $k$-coloured graph on at least $N + \eps n$ vertices, where every vertex has at most $\delta n$ non-neighbours, and there is no monochromatic connected matching on at least $\alpha n$ vertices. 
		Then there exists a $k$-colouring $G'$ of $K_N$ that contains no monochromatic connected matching on at least $\alpha n$ vertices. Moreover, there exists such $G'$ that contains an induced subgraph of $G$ on $N$ vertices.
	\end{thm}

	We now state the main result in this section, generalising \Cref{thm:matching-complete}. This theorem is applicable for both connected matchings and connected $2$-matchings.

	\begin{thm} \label{thm:matching-general}
		Let $s$, $k_0 \le k$ be integers, let $\beta \ge \alpha > 0$, and let $\eps > 0$ be sufficiently small. Set $\delta = 2 \cdot \big( \frac{56s^3\alpha}{\beta} \big)^{2k}$, and let $N_1, \ldots, N_s \le \beta n$ and $\alpha_1, \ldots, \alpha_k \ge \alpha$.
		Let $F$ be a graph on vertex set $[s]$ (possibly with loops), and denote $\G = F(N_1 + \eps n, \ldots, N_s + \eps n)$ and $\G' = F(N_1, \ldots, N_s)$.

		Suppose that $G$ is a $k$-coloured spanning subgraph of $\G$ such that $d_G(u) \ge d_{\G}(u) - \delta n$ for every vertex $u$; there is no $\ell$-coloured connected ($2$-)matching on at least $\alpha_{\ell} n$ vertices for $\ell \in [k_0]$; and there is no $\ell$-coloured connected ($2$-)matching on at least $\alpha_{\ell} n$ vertices contained in a non-bipartite $\ell$-coloured component for $\ell \in [k_0 + 1, k]$.

		Then there exists a $k$-colouring $G'$ of $\G'$ such that there is no $\ell$-coloured connected ($2$-)matching on at least $\alpha_{\ell} n$ vertices for $\ell \in [k_0]$; and there is no $\ell$-coloured connected ($2$-)matching on at least $\alpha_{\ell} n$ vertices contained in a non-bipartite $\ell$-coloured component for $\ell \in [k_0 + 1, k]$. Moreover, there exists such $G'$ which contains a induced subgraph of $G$ with $N_i$ vertices from the set replacing vertex $i$ in $F$ for $i \in [s]$.
	\end{thm}

	\begin{rem} \label{rem:multigraphs}
		It is sometimes convenient to allow edges to have multiple colours (see, e.g., \cite{balogh-et-al-b,gyarfas-et-al}).
		One can think of the graphs in the above statements as being multicolour. We even use multicoloured graphs in the proofs of these statements.
	\end{rem}
	
	The following lemma formalises {\L}uczak's connected matchings method in a fairly general form. It generalises Lemma 3 in \cite{figaj-luczak} and can be proved similarly; we do not include a proof here.
	
	\begin{lem}[Generalisation of Lemma 3 in \cite{figaj-luczak}] \label{lem:figaj-luczak-general}
		Let $s$, $k_0 \le k$ be integers, let $\alpha_1, \ldots, \alpha_k, \beta_1, \ldots, \beta_s > 0$, and let $F$ be a graph on vertex set $[s]$ (possibly with loops).
		
		Suppose that for every $\eps > 0$ there exists $\delta > 0$ such that for every large $n$, if $G$ is a spanning subgraph of $\G := F((\beta_1 + \eps)n, \ldots, (\beta_s + \eps)n)$ with $e(G) \ge (1 - \delta) e(\G)$, the following holds: every $k$-colouring of $G$ has an $\ell$-coloured connected $2$-matching on at least $\alpha_{\ell} n$ vertices for some $\ell \in [k]$, which is contained in a non-bipartite $\ell$-coloured component if $\ell \in [k_0+1, k]$.

		Then for every $\eps > 0$ there exists $T$ such that for sufficiently large $n$, every $k$-colouring of $F((\beta_1 + \eps)n, \ldots, (\beta_s + \eps)n)$ satisfies the following: either  there is an $\ell$-coloured $C_t$ for every even $t \in [T, \alpha_{\ell}n]$ for some $\ell \in [k_0]$, or  there is an $\ell$-coloured $C_t$ for every integer $t \in [T, \alpha_{\ell} n]$ for some $\ell \in [k_0 + 1, k]$.
	\end{lem}

	Finally, we turn to the proof of \Cref{thm:main-general}, which follows easily from \Cref{lem:figaj-luczak-general,thm:matching-general}.

	\begin{proof}[Proof of \Cref{thm:main-general}] 
		Let $\eps > 0$; without loss of generality, we assume that $\eps$ is sufficiently small. Let $\delta > 0$ be sufficiently small.
		Denote $\G_n = F((\beta_1 + \eps)n, \ldots, (\beta_s + \eps)n)$ and $\G_n' = F((\beta_1 + \eps/4)n, \ldots, (\beta_s + \eps/4)n)$.
		By the assumption of \Cref{thm:main-general}, for large enough $n$, every $k$-colouring of $\G_n'$ has an $\ell$-coloured connected $2$-matching on at least $\alpha_{\ell} n$ vertices, which is contained in a non-bipartite $\ell$-coloured component if $\ell \in [k_0 + 1, k]$. 
		
		 By \Cref{lem:figaj-luczak-general}, it suffices to show that, for large $n$, if $G$ is a $k$-colouring of a spanning subgraph of $\G_n$ such that $e(G) \ge (1 - \delta) e(\G_n)$, then $G$ contains an $\ell$-coloured connected matching on at least $\alpha_{\ell} n$ vertices, which is contained in an $\ell$-coloured non-bipartite component if $\ell \in [k_0+1, k]$. Suppose this is not the case and fix suitable large $n$ and graph $G$. 

		Let $W$ be the set of vertices $w$ in $G$ with $d_G(w) \le d_{\G_n}(w) - \sqrt{\delta} n$. As $|E(\G_n) \setminus E(G)| \le \delta |\G_n|^2/2 \le \delta (2s\beta n)^2/2$, where $\beta = \max\{\beta_1, \ldots, \beta_s\}$ (using that $\eps$ is small), we have $|W| \le \frac{2}{\sqrt{\delta}n} \cdot |E(\G_n) \setminus E(\G_n')| \le (\eps/2) n$ (using that $\delta$ is small). Define $G' = G \setminus W$. By \Cref{thm:matching-general} (with parameters $\eps' = \eps / 4$, $\beta_i' = \beta_i + \eps/4$ for $i \in [s]$, $\delta' = \sqrt{\delta}$), there is a $k$-colouring of $\G_n'$ with no $\ell$-coloured connected matching on at least $\alpha_{\ell} n$ vertices, which is contained in a non-bipartite $\ell$-coloured component if $\ell \in [k_0 + 1, k]$, a contradiction to the assumption on $\G_n'$.
	\end{proof}

	\subsection{Stability} \label{subsec:stability}

		It is often the case that in order to determine the Ramsey number of paths or cycles precisely, one should prove a \emph{stability} result: either an edge-coloured graph contains a large monochromatic connected matching, or it has a special structure. Our method is handy also for proving stability results.

		Suppose that we are given an edge-coloured graph $H$ where we wish to find a monochromatic path or cycle of given length. 
		Apply the regularity lemma and let $G$ be the so-called \emph{reduced graph} (whose vertices represent the clusters in the regular partition, and edges represent regular pairs and are coloured suitably: e.g.\ by a majority colour, or by all colours with large enough density in the pair). {\L}uczak's method implies that a monochromatic connected ($2$-)matching in $G$ on $\alpha |G|$ vertices can be lifted to a monochromatic cycle in $H$ (of the same colour) on at least $(\alpha - \eps)|H|$ vertices (where $\eps > 0$ can be arbitrarily small). Thus, a stability result would tell us that either there is a suitably long monochromatic path or cycle in $H$, or $G$ has a certain special structure, from which it follows that $H$ has a special structure. In the latter case, the required monochromatic path or cycle can be found `by hand' by looking at the structure of $H$ more closely.

		If the original graph $H$ is a complete graph, then the reduced graph $G$ is almost complete, and similarly if $H$ is a blow-up of $F$ then $G$ is an almost blow-up of $F$. Assuming that $G$ does not contain a large enough monochromatic connected ($2$-)matching, by applying \Cref{thm:matching-general}, we obtain a slightly smaller complete graph (or blow-up of $F$) $G'$ without a large enough monochromatic connected ($2$-)matching. Moreover, $G'$ is very similar to $G$, meaning that $E(G) \triangle E(G')$ is small. Given a stability result for complete graphs (or blow-ups of $F$), it follows that $G'$ has a special structure, which implies that $G$ has a similar structure, as required.
		This is illustrated at the end of \Cref{sec:three-colours}.

\section{The Gallai-Edmonds decomposition} \label{sec:prelims}

	We shall use the \emph{Gallai-Edmonds decomposition} (which we abbreviate to \emph{GE-decomposition}, following \cite{balogh-et-al}) of a graph $G$, defined next.

	\begin{defn} \label{def:ge}
		In a graph $G$, let $B$ be the set of vertices that are covered by every maximum matching in $G$, let $A$ be the set of vertices in $B$ that have a neighbour outside of $B$, and set $C := B \setminus A$ and $D := V(G) \setminus B$. The \emph{GE-decomposition} of $G$ is the partition $\{A, C, D\}$ of $V(G)$.
	\end{defn}	

	The following theorem, due to Edmonds and Gallai, lists useful properties of the GE-decomposition.
	Recall that a graph $H$ is called \emph{factor-critical} if for every vertex $u$ in $H$ the graph $H \setminus \{u\}$  has a perfect matching.

	\begin{thm}[Edmonds \cite{edmonds} and Gallai \cite{gallai}; see also Theorem 3.2.1 in \cite{lovasz-plummer}] \label{thm:ge}
		Let $\{A, C, D\}$ be the GE-decomposition of a graph $G$, as given in \Cref{def:ge}. Then
		\begin{enumerate} [label = \rm(M\arabic*)]
			\item \label{itm:ge-AC}
				Every maximum matching in $G$ covers $C$ and matches $A$ into distinct components of $G[D]$.
			\item \label{itm:ge-factor-critical}
				Every component of $G[D]$ is factor-critical. 
		\end{enumerate}
	\end{thm}

	The next lemma will be very useful in our proofs; here a complete blow-up of a graph $F$ is the blow-up of $F$ obtained by replacing each vertex of $F$ be a (possibly empty) complete graph.

	\begin{lem} \label{lem:ge}
		Let $G$ be a maximal graph on $n$ vertices without a matching of size $m$. Then $G$ is a complete blow-up of a star.
	\end{lem}

	\begin{proof} 
		Let $\{A, C, D\}$ be the GE-decomposition of $G$, and let $D_1, \ldots, D_r$ be the vertex sets of the components of $G[D]$. Define $H$ to be the graph with vertex set $V(G)$ whose edges are all pairs of vertices that either touch $A$ or are contained in one of the sets $C, D_1, \ldots, D_r$. Then $G \subseteq H$ and $H$ is a complete blow-up of a star on $r+2$ vertices. 

		Note that by \Cref{thm:ge}, every maximum matching $M$ of $G$ covers exactly $n - (r - |A|)$ vertices (as $M$ consists of a perfect matching in $G[C]$, a matching in $G[D]$ that covers all but exactly one of the vertces in $D_i$ for $i \in [r]$, and a matching between $A$ and $D$ that covers $A$). Now consider a matching $M'$ in $H$. Because each $D_i$ is odd (by \ref{itm:ge-factor-critical}), either $M'$ leaves one vertex of $D_i$ uncovered, or it contains at least one edge between $D_i$ and $A$. It follows that at least $r - |A|$ vertices of $D$ are uncovered by $M'$, and so $M'$ covers at most $n - (r - |A|)$ vertices. As $G \subseteq H$, it follows that $G$ and $H$ have the same matching number, so by maximality of $G$, we have $G = H$. This completes the proof of \Cref{lem:ge}, since $H$ is a complete blow-up of a star.
	\end{proof}

	Recall that a \emph{$2$-matching} is a collection of vertex-disjoint edges and odd cycles and its \emph{order} is the number of vertices it covers.
	Here is a result of Pulleyblank \cite{pulleyblank} which provides a variant of \Cref{thm:ge} for $2$-matchings (this statement is a simplified version of Pulleyblank's result, mirroring Theorem 4.1 in \cite{debiasio-lo}).

	\begin{thm} \label{thm:pulleyblank}
		Let $D$ be the set of vertices in a graph $G$ that are left uncovered by at least one maximum matching in $G$, and let $D'$ be the set of vertices in $D$ that are isolated in $G[D]$. If $M$ is a maximum $2$-matching in $G$ for which the number of vertices contained in odd cycles in minimised, then $M$ covers $V(G) \setminus D'$ and the edges of $M$ incident with vertices in $D'$ induce a matching that covers $N(D')$ (the neighbourhood of $D'$). 
	\end{thm}

	Here is a version of \Cref{lem:ge} for $2$-matchings.

	\begin{lem} \label{lem:ge-$2$}
		Let $G$ be a maximal graph on $n$ vertices without a $2$-matching on $m$ vertices. Then $G$ is a blow-up of a star where all but at most one leaf are replaced by a single vertex.
	\end{lem}

	\begin{proof}
		Let $D'$ be as in the statement of \Cref{thm:pulleyblank}, let $A' = N(D')$, and let $C' = V(G) \setminus (A' \cup D')$. Let $H$ be the graph on $V(G)$ whose edges are pairs of vertices that either touch $A'$ or are contained in $C'$. Then $G \subseteq H$ and $H$ is a complete blow-up of a star, where all but at most one leaf are replaced by a single vertex. By \Cref{thm:pulleyblank}, every maximum $2$-matching in $G$ covers exactly $n - (|D'| - |A'|)$ vertices. It is easy to see that the same can be said for $H$, and thus by maximality of $G$ we have $G = H$, completing the proof.
	\end{proof}

	Here is another consequence of \Cref{thm:ge}, which we will use in \Cref{sec:three-colours}.

	\begin{cor} \label{cor:ge-remove-vx}
		Let $G$ be a connected graph on $n \ge 2m+2$ vertices that does not have a matching of size $m+1$. Then there exists a vertex $u$ such that $G \setminus \{u\}$ does not have a matching of size $m$. 
	\end{cor}

	\begin{proof}
		Let $B$ be the set of vertices in $G$ that are contained in every maximum matching in $G$. 
		If $B = \emptyset$, then, as $G$ is connected, \Cref{thm:ge}~\ref{itm:ge-factor-critical} implies that $G$ is factor critical, i.e.\ $G \setminus \{u\}$ has a perfect matching for every vertex $u$ in $G$. But this implies that $G \setminus \{u\}$ has a matching of size at least $(n-1)/2 > m$, a contradiction. 
		It follows that $B \neq \emptyset$. Take $u$ to be any vertex in $B$, and note that the matching number of $G \setminus \{u\}$ is smaller than the matching number of $G$, as required.
	\end{proof}

\section{Proofs} \label{sec:proofs}

	Our main aim in this section is to prove \Cref{thm:matching-general}. 
	We start with the proof of \Cref{thm:matching-complete}, a special case of \Cref{thm:matching-general}. The two theorems are proved similarly, but the proof of the former avoids some technicalities (and is sufficient for $k$-colour Ramsey numbers of paths and cycles), so we hope that including its proof first will be instructive.

	\subsection{Proof of the simplified version}

		\begin{proof}[Proof of \Cref{thm:matching-complete}]
			Recall that $G$ is a $k$-coloured graph on at least $N + \eps n$ vertices, where $N \le \beta n$; every vertex has at most $\delta n$ non-neighbours; and there is no monochromatic connected matching on at least $\alpha n$ vertices. Without loss of generality, we assume that $G$ has exactly $N + \eps n$ vertices. We may assume that $\eps > 0$ is small, and we have that $\delta$ is sufficiently small and $\beta \ge \alpha$.

			Let $G_1$ be a $k$-multicoloured graph (so edges can have more than one colour) on $V(G)$ such that $G \subseteq G_1$ (here we take the colours into account, meaning that if $e$ is an edge in $G$ of colour $c$ then $e$ has colour $c$ in $G_1$, but it might have other colours too); $G_1$ does not contain a monochromatic connected matching on at least $\alpha n$ vertices; and $G_1$ is maximal with respect to these properties. We claim that $G_1$ satisfies the following three properties.
			
			\begin{enumerate} [label = \rm(S\arabic*)]
				\item \label{itm:G1-dense}
					Every vertex in $G_1$ has at most $\delta n$ non-neighbours.
				\item \label{itm:G1-no-large-mono-cm}
					There is at most one $\ell$-coloured component on fewer than $\alpha n / 2$ vertices, for $\ell \in [k]$. It follows that there are at most $\frac{|G|}{\alpha n / 2} + 1 \le \frac{2(\beta + \eps)}{\alpha} + 1 \le \frac{4\beta}{\alpha}$ components of colour $\ell$ (here we used $\beta \ge \alpha$ and that $\eps$ is small).
				\item \label{itm:G1-star}
					Every monochromatic component $U$ is a complete blow-up of a star. 
					Denote the vertices at the centre of the star by $H(U)$ (for `head') and the remaining vertices in $U$ by $T(U)$ (for `tail').
					Note that if the star is a single vertex or edge, then $U$ is a clique. In this case any choice of $H(U)$ and $T(U)$ that partition the vertex set of $U$ is acceptable.
			\end{enumerate}
			Indeed, \ref{itm:G1-dense} follows as the same holds for $G$, and $G_1$ contains $G$ as a subgraph and has the same vertex set. To see \ref{itm:G1-no-large-mono-cm}, suppose that there are two $\ell$-coloured components $U$ and $W$, each on fewer than $\alpha n / 2$ vertices. Then adding an $\ell$-coloured edge between them would yield a graph that contains $G_1$ and does not have a monochromatic connected matching on at least $\alpha n$ vertices, for $\ell' \in [k]$, contradicting the maximality of $G_1$. Finally, \ref{itm:G1-star} follows from \Cref{lem:ge}.

			Denote by $\comp{G}_1$ the \emph{complement} of $G_1$, namely the graph on $V(G_1)$ whose edges are non-edges of $G_1$.
			\begin{claim} \label{claim:matching-complement}
				$\comp{G}_1$ does not have a matching of size larger than $\big(\frac{16\beta}{\alpha}\big)^k \cdot \delta n$.
			\end{claim}

			\begin{proof}
				We assign a \emph{type} $\tau(u) = (U_1, \ldots, U_k, \mu_1, \ldots, \mu_k)$ to each vertex $u$ in $G_1$, so that $U_{\ell}$ is the $\ell$-coloured component that contains $u$ (observe that there is a unique such component $U_{\ell}$, which is a single vertex when $u$ is not incident with edges of colour $\ell$), $\mu_{\ell} = H$ if $u \in H(U_{\ell})$ and $\mu_{\ell} = T$ otherwise. By \ref{itm:G1-no-large-mono-cm}, the total number of types, denoted $t$, satisfies $t \le \big( \frac{4\beta}{\alpha} \big)^k \cdot 2^k = \big(\frac{8\beta}{\alpha}\big)^k$. 

				Suppose that $M$ is a matching in $\comp{G}_1$ with $|M| > \big(\frac{16\beta}{\alpha}\big)^{2k} \cdot \delta n$. Then there exists two types $\tau = (U_1, \ldots, U_k, \mu_1, \ldots, \mu_k)$ and $\sigma = (W_1, \ldots, W_k, \nu_1, \ldots, \nu_k)$ and a submatching $M_0 \subseteq M$ such that $|M_0| > |M| / t^2 \ge 4^k \delta n$ and the edges in $M_0$ have ends of types $\tau$ and $\sigma$. Denote by $X_0$ and $Y_0$ the vertices in $V(M_0)$ of types $\tau$ and $\sigma$, respectively.

				We will find submatchings $M_0 \supseteq M_1 \supseteq \ldots \supseteq M_k$ such that $|M_{\ell}| \ge \frac{1}{4} \cdot |M_{\ell-1}|$ and there are no $\ell$-coloured edges between $X_{\ell} := X_0 \cap V(M_{\ell})$ and $Y_{\ell} := Y_0 \cap V(M_{\ell})$. Assuming such submatchings are found, then $|X_{k}| = |Y_{k}| = |M_k| \ge 4^{-k} |M_0| > \delta n$, and there are no $\ell$-coloured edges between $X_k$ and $Y_k$ for $\ell \in [k]$, implying that there are no edges of $G_1$ between $X_k$ and $Y_k$. This is a contradiction to \ref{itm:G1-dense}. It thus remains to show that such matchings $M_1, \ldots, M_k$ exist.

				Suppose that $M_0 \supseteq M_1 \supseteq \ldots \supseteq M_{\ell - 1}$, where $\ell \in [k]$, satisfy the above requirements. We will show how to obtain a suitable $M_{\ell}$. To do so, consider two cases: $U_{\ell} \neq W_{\ell}$ and $U_{\ell} = W_{\ell}$. In the former case, vertices of type $\tau$ are type $\sigma$ are contained in distinct $\ell$-coloured components. In particular, there are no $\ell$-coloured edges between $X_{\ell-1}$ and $Y_{\ell-1}$, so taking $M_{\ell} = M_{\ell-1}$ would suffice. 

				Now consider the latter case, where $U_{\ell} = W_{\ell}$, i.e.\ vertices of type $\tau$ and $\sigma$ are contained in the same $\ell$-coloured component $U := U_{\ell}$. We claim that $\mu_{\ell} = \nu_{\ell} = T$, i.e.\ all vertices of types $\tau$ and $\sigma$ are in $T(U)$. Indeed, if say $\mu_{\ell} = H$ then the vertices in $X_0$ are adjacent in colour $\ell$ to all other vertices in $U$ (as vertices in $H(U)$ are joined to all other vertices in $U$). In particular, for every $x \in X_0$ and $y \in Y_0$, the pair $xy$ is an $\ell$-coloured edge in $G_1$, contradicting the assumption that there is a perfect matching of non-edges between $X_0$ and $Y_0$.

				Recall that the $\ell$-coloured edges in $G_1[T(U)]$ form a disjoint union of cliques; denote these cliques by $K_1, \ldots, K_t$. Then all edges in $M_{\ell-1}$ have ends in distinct cliques $K_1, \ldots, K_t$, as edges in $M_{\ell-1}$ are non-edges in $G_1$. Let $\{A, B\}$ be a random partition of $[t]$, and define
				\begin{equation*}
					M_{\ell} = \left\{xy \in M_{\ell-1} :\,\, x \in X_{\ell-1} \cap \bigcup_{i \in [A]}K_i \text{ and } y \in Y_{\ell-1} \cap \bigcup_{i \in [B]} K_i\right\}.
				\end{equation*}
				Note that for every edge $e \in M_{\ell-1}$, the probability that $e$ is in $M_{\ell}$ is $1/4$. Indeed, if $e = xy$, where $x \in X_{\ell-1} \cap K_i$ and $y \in Y_{\ell-1} \cap K_j$ (so $i$ and $j$ are distinct elements in $[k]$), then $\Pr[e \in M_{\ell}] = \Pr[i \in A] \cdot \Pr[j \in B] = 1/4$. It follows that $\Ex[|M_{\ell}|] = \frac{1}{4} \cdot |M_{\ell-1}|$, and so there is a choice of $A$ and $B$ such that $|M_{\ell}| \ge \frac{1}{4} \cdot |M_{\ell-1}|$. Take such an $M_{\ell}$, and observe that there are no $\ell$-coloured edges between $X_{\ell}$ and $Y_{\ell}$. Thus $M_{\ell}$ satisfies the requirements, completing the proof of the claim.
			\end{proof}

			Let $M$ be a maximum matching in $\comp{G}_1$. By \Cref{claim:matching-complement}, $|M| \le \big(\frac{16 \beta}{\alpha}\big)^{2k} \cdot \delta n = (\eps / 2) \cdot n$. Consider the graph $G_2 = G_1 \setminus V(M)$. This is a $k$-multicoloured complete graph, as the existence of a non-edge in $G_2$ would imply the existence of a larger matching in $\comp{G}_1$ than $M$, contrary to the choice of $M$. By the upper bound on the size of $M$, we have $|G_2| \ge |G_1| - 2|M| \ge |G_1| - \eps n \ge N$. \Cref{thm:matching-complete} follows by taking $G'$ to be any induced subgraph of $G_2$ on $N$ vertices.
		\end{proof}

	\subsection{Proof of the general version}
		We now prove \Cref{thm:matching-general}. The proof is similar to the above proof, so we allow ourselves to skip some details.
		The statement of \Cref{thm:matching-general}, as well as its proof, hold for both matchings and $2$-matchings simultaneously.

		\begin{proof}
			Recall that $F$ is a graph on vertex set $[s]$ (possibly with loops) and $\G = F(N_1 + \eps n, \ldots, N_s + \eps n)$ (where $N_1, \ldots, N_s \le \beta n$). Recall also that $G$ is a $k$-coloured spanning subgraph of $\G$ such that 
			\begin{enumerate} [label = \rm(P\arabic*)]
				\item \label{itm:G-deg}
					$d_G(u) \ge d_{\G}(u) - \delta n$ for every vertex $u$ in $G$,
				\item \label{itm:G-first-colours}
					there is no $\ell$-coloured ($2$-)matching on at least $\alpha_{\ell}n$ vertices,
				\item \label{itm:G-last-colours}
					there is no $\ell$-coloured ($2$-)matching on at least $\alpha_{\ell}n$ vertices which is contained in a non-bipartite $\ell$-coloured component for $\ell \in [k_0 + 1, k]$.
			\end{enumerate}

			Let $G_1$ be a $k$-multicoloured graph such that $G \subseteq G_1 \subseteq \G$ which satisfies \ref{itm:G-first-colours} and \ref{itm:G-last-colours} above and is maximal with respect to these properties.
			We claim that $G_1$ satisfies the following properties.
			\begin{enumerate} [label = \rm(G\arabic*)]
				\item \label{itm:G1-general-deg}
					$d_{G_1}(u) \ge d_{\G}(u) - \delta n$ for every vertex $u$ in $G_1$.
				\item \label{itm:G1-general-num-compts}
					There are at most $e(F)$ components of colour $\ell$ that have at least two and fewer than $\alpha_{\ell}n/2$ vertices. It follows that the number of $\ell$-coloured components that consist of more than one vertex is at most $\frac{|\G|}{\alpha n / 2 - 1} + e(F) \le \frac{s(\beta + \eps)n}{\alpha n / 3} + s^2 \le \frac{6s^2 \beta}{\alpha}$ (using $\beta \ge \alpha \ge \alpha_{\ell}$, $\eps$ being small, and $n$ being large).
				\item \label{itm:G1-general-struct-compts}
					Every $\ell$-coloured component $U$ is either the intersection of a complete blow-up of a star $S$ with $\G$, or the intersection of a complete bipartite graph with $\G$, where the latter can only hold if $\ell \in [k_0 + 1, k]$.
					In the former case, write $H(U)$ for the vertices in $U$ in the blow-up of the centre of the star $S$ and $T(U)$ for the remaining vertices. In the latter case, denote the bipartition of $U$ by $\{L(U), R(U)\}$ (for `left' and `right').
			\end{enumerate}
			Property \ref{itm:G1-general-deg} follows as $G \subseteq G_1 \subseteq \G$. To see \ref{itm:G1-general-num-compts}, suppose that there are more than $e(F)$ components of colour $\ell$ of order in $[2, \alpha n / 2)$. Since each such component contains an edge of $\G$ which corresponds to an edge of $F$, there exist two distinct $\ell$-coloured components $U$ and $W$ and an edge $xy$ in $F$ such that both $U$ and $W$ contain an edge between the blobs corresponding to $x$ and $y$ in $\G$; let $x_u y_u$ and $x_w y_w$ be such edges (where $x_u, x_w$ are in the blow-up of $x$). Now form $G_1'$ by adding an $\ell$-coloured edge between $x_u$ and $y_w$. Then $G_1 \subsetneq G \subseteq \G$, and $G_1'$ satisfies \ref{itm:G-first-colours} and \ref{itm:G-last-colours} above, contrary to the maximality of $G_1$. 

			For \ref{itm:G1-general-struct-compts}, let $U$ be an $\ell$-coloured component. Then either $U$ has no ($2$-)matching on at least $\alpha_{\ell} n$ vertices, or $U$ is bipartite and $\ell \in [k_0 + 1, k]$. If the former holds, let $H$ be a maximal graph that contains $U$, has the same vertex set, and has no ($2$-)matching on at least $\alpha_{\ell} n$ vertices. Then by \Cref{lem:ge} (for matchings) or \Cref{lem:ge-$2$} (for $2$-matchings), $H$ is a complete blow-up of a star. Property \ref{itm:G1-general-struct-compts} follows from the maximality of $G_1$. Similarly, if $U$ is bipartite, Property \ref{itm:G1-general-struct-compts} follows by considering a complete bipartite graph that contains $U$ and has the same vertex set as $U$.

			Let $\comp{G}_1$ be the graph on $V(G)$ with edges $E(\G) \setminus E(G_1)$. 
			\begin{claim} \label{claim:general-matching-complement}
				$\comp{G}_1$ does not have a matching of size larger than $\big( \frac{56s^3 \beta}{\alpha} \big)^{2k} \cdot \delta n$.
			\end{claim}

			\begin{proof}
				For $\ell \in [k]$, let $I_{\ell}$ be the set of vertices in $G_1$ that are not incident with $\ell$-coloured edges.

				Define the \emph{type} $\tau(u)$ of a vertex $u$ to be $\tau(u) = (U_1, \ldots, U_k, \mu_1, \ldots, \mu_k, v)$, where $U_{\ell}$ and $\mu_{\ell}$ are defined as follows for $\ell \in [k]$
				\begin{itemize}
					\item
						If $u \in I_{\ell}$ then $U_{\ell} = I_{\ell}$. Otherwise, $U_{\ell}$ is the $\ell$-coloured component in $G_1$ that contains $u$.
					\item
						If $u \in I_{\ell}$ set $\mu_{\ell} = I$; otherwise take $\mu_{\ell}$ to be one of $H, T, L, R$ according to which of the four sets $H(U_{\ell})$, $T(U_{\ell})$, $L(U_{\ell})$, $R(U_{\ell})$ the vertex $u$ belongs to (note that only two of these sets are defined, in \ref{itm:G1-general-struct-compts}).
					\item
						The vertex $v$ is in $F$ and $u$ is in the blob in $\G$ corresponding to $v$.
				\end{itemize}
				Observe that the total number of types, denoted $t$, satisfies $t \le \left( \big( \frac{6s^2 \beta}{\alpha} + 1\big)^k\right) \cdot 4^k \cdot s \le  \big( \frac{28s^3\beta}{\alpha} \big)^{k}$ (crudely), using \ref{itm:G1-general-num-compts}. Thus, if $M$ is a matching in $\comp{G}_1$ of size larger than $\big( \frac{56s^3\beta}{\alpha} \big)^{2k}$, then there are two types, $\tau = (U_1, \ldots, U_k, \mu_1, \ldots, \mu_k, v)$ and $\sigma = (W_1, \ldots, W_k, \nu_1, \ldots, \nu_k, w)$, and a submatching $M_0 \subseteq M$, such that $|M_0| \ge |M|/t^2 > 4^k \delta n$ and the edges in $M_0$ have ends of types $\tau$ and $\sigma$. Denote by $X_0$ and $Y_0$ the vertices of type $\tau$ and $\sigma$, respectively, in $V(M_0)$.
				Note that since $M$ is a matching in $\G$, we have that $vw$ is an edge in $F$. It follows that $X_0$ is fully joined to $Y_0$ in $\G$. 

				We will find submatchings $M_0 \supseteq M_1 \supseteq \ldots \supseteq M_k$ such that $|M_{\ell}| > \frac{1}{4} \cdot |M_{\ell-1}|$ and there are no $\ell$-coloured edges between $X_{\ell} := X_0 \cap V(M_{\ell})$ and $Y_{\ell} := Y_0 \cap V(M_{\ell})$, for $\ell \in [k]$. Given such submatchings, the sets $X_k$ and $Y_k$ have size larger than $\delta n$ and there are no edges of $G_1$ between them. This is a contradiction to \ref{itm:G1-general-deg}, because, as explained above, $X_k$ is fully joined to $Y_k$ in $\G$.

				Suppose that $M_0 \supseteq M_1 \supseteq \ldots \supseteq M_{\ell-1}$ are as above and $\ell \in [k]$.
				We consider three cases: $U_{\ell} \neq W_{\ell}$; $U_{\ell} = W_{\ell} = I_{\ell}$; and $U_{\ell} = W_{\ell} \neq I_{\ell}$.
				In the first two cases there are no $\ell$-coloured edges between vertices of type $\tau$ and vertices of type $\sigma$; we can thus take $M_{\ell} = M_{\ell-1}$. 

				It remains to consider the third case; denote $U := U_{\ell} = W_{\ell}$. Then either $\mu_{\ell}, \nu_{\ell} \in \{H, T\}$, or $\mu_{\ell}, \nu_{\ell} \in \{L, R\}$. In fact, either $\mu_{\ell} = \nu_{\ell} = T$ or $\mu_{\ell} = \nu_{\ell} \in \{R, L\}$, because otherwise $X_0$ is fully joined to $Y_0$ in colour $\ell$, contradicting the assumption that there is a perfect matching of non-edges between $X_0$ and $Y_0$. If the latter holds, we can take $M_{\ell} = M_{\ell-1}$ because there are no $\ell$-coloured edges within $R(U)$ or $L(U)$, so suppose that the former holds. Recall that $T(U)$ is the intersection of a disjoint union of cliques with $\G$; denote the cliques by $K_1, \ldots, K_t$. Let $\{A, B\}$ be a random partition of $[t]$, and define
				\begin{equation*}
					M_{\ell} = \left\{ xy \in M_{\ell-1} : \,\, x \in X_{\ell-1} \cap \bigcup_{i \in A} K_i \text{ and } y \in Y_{\ell-1} \cap \bigcup_{i \in B} K_i \right\}.
				\end{equation*}
				It is easy to see that $\Ex[|M_{\ell}|] = \frac{1}{4} \cdot |M_{\ell-1}|$, and so there exists $A, B$ such that $|M_{\ell}| \ge \frac{1}{4} \cdot |M_{\ell-1}|$. The matching $M_{\ell}$ satisfies the requirements.
			\end{proof}

			Let $M$ be a maximum matching in $\comp{G}_1$. By \Cref{claim:general-matching-complement}, we have $|M| \le \big( \frac{56s^3 \beta}{\alpha} \big)^{2k} \cdot \delta n \le \eps n / 2$. Let $G_2 = G_1 \setminus V(M)$. Then $G_2$ is a $k$-coloured graph satisfying \ref{itm:G-first-colours} and \ref{itm:G-last-colours} that contains $\G' = F(N_1, \ldots, N_s)$ as a subgraph. 
			\Cref{thm:matching-general} follows by taking $G'$ to be any copy of $\G'$ in $G_2$. 
		\end{proof}

\section{Three colours} \label{sec:three-colours}

	As mention in the introduction, the $3$-colour Ramsey number of $P_n$ is known for sufficiently large $n$ (see Gy\'arf\'as, Ruszink\'o, S\'ark\"ozy and Szemer\'edi \cite{gyarfas-et-al} for the exact result and Figaj and {\L}uczak \cite{figaj-luczak} for an asymptotic result). In particular, it is known that $r_3(P_n) = (2 + o(1))n$. In this section we give an alternative proof of the upper bound $r_3(P_n) \le (2 + o(1))n$, as well as a sketch of a proof of a corresponding stability result. Our proof is arguably simpler than the proofs in \cite{figaj-luczak,gyarfas-et-al}, and uses a different method, namely induction. The main step in our proof is covered by the following lemma. Throughout this section, we use the notation $\cm(m)$ to denote a connected matching of size $m$.

	\begin{lem} \label{lem:three-colours}
		Let $G$ be a $3$-coloured $K_{4m+1}$. If $G$ does not have a monochromatic $\cm(m+1)$, then the following holds, up to relabelling of the colours. 
		\begin{enumerate} [label = \rm(C)]
			\item \label{itm:three-extremal}
				there is a partition $\{X, Y, Z, W\}$ of $V(G)$ such that $|X| = |Y| = m$; $|Z| \le m$; the edges in $[X, Z]$ and $[Y, W]$ are red; the edges in $[X, W]$ and $[Y, Z]$ are blue; the edges in $[Z, W]$ are green; if $|W| \ge m + 2$, then all edges in $W$ are green; and if $Z \neq \emptyset$, then the edges in $[X, Y]$ are also green.
		\end{enumerate}
	\end{lem}

	The above lemma implies the required upper bound on $r_3(P_n)$ using \Cref{cor:main-complete-symmetric}.

	\begin{proof}[Alternative proof of $r_3(P_n) \le (2 + o(1))n$]
		By \Cref{lem:three-colours}, every $3$-colouring of $K_{4m+1}$ contains a monochromatic $\cm(m)$ for $m \ge 1$ (indeed, either there is a monochromatic $\cm(m+1)$ or \ref{itm:three-extremal} holds, in which case there is a blue $\cm(m)$ in $G[X, W]$). In particular, for any $\eps > 0$ and sufficiently large $n$, every $3$-colouring of $K_{\ceil{(1 + \eps)2n}}$ contains a monochromatic connected matching on at least $n$ vertices. \Cref{cor:main-complete-symmetric} implies that $r_3(P_n) \le (2 + o(1))n$.
	\end{proof}

	To prepare for the proof of \Cref{lem:three-colours}, we need the following lemma regarding $2$-colourings of complete bipartite graphs.

	\begin{lem} \label{lem:bip}
		Let $m \ge 1$ and let $G$ be a complete bipartite graph with bipartition $\{A, B\}$ such that $|A| = 2m$ and $|B| = 2m+1$. Then for every $2$-colouring of $G$, one of the following holds.
		\begin{enumerate}[label = \rm(B\arabic*)]
			\item \label{itm:bip-cm}
				there is a monochromatic $\cm(m+1)$,
			\item \label{itm:bip-extremal}
				there are partitions $\{X, Y\}$ of $A$ and $\{Z, W\}$ of $B$ such that: $|X| = |Y| = m$; all $[X, Z]$ and $[Y, W]$ edges are red; and all $[X, W]$ and $[Y, Z]$ edges are blue. Note that one of $Z$ and $W$ might be empty.
		\end{enumerate}
	\end{lem}

	\begin{proof}
		Consider a red-blue colouring of $G$ with no monochromatic $\cm(m+1)$.

		Suppose first that there is a red component that contains all red edges. Then any maximum red matching has size at most $m$ (by assumption on $G$), so by K\"onig's theorem there is a set $U$ of size at most $m$ that covers all red edges. In particular, all $[A \setminus U, B \setminus U]$ edges are blue. It follows that $U \subseteq A$, because otherwise $|A \setminus U|, |B \setminus U| \ge m+1$ and $[A \setminus U, B \setminus U]$ contains a blue $\cm(m+1)$. 
		Take $X = U$, $Y = A \setminus U$, $Z = B$ and $W = \emptyset$ to see that \ref{itm:bip-extremal} holds.

		Now suppose that there is no red component that contains all red edges, nor a blue component that contains all blue edges. It follows that there are partitions $\{X, Y\}$ of $A$ and $\{Z, W\}$ of $B$, such that $X, Y, Z, W \neq \emptyset$ and all $[X, W]$ and $[Y, Z]$ edges are blue. By the assumption on the blue edges, all $[X, Z]$ and $[Y, W]$ edges are red. Without loss of generality $|X| \ge m$ and $|W| \ge m+1$. By assumption on $G$, we have $|X| = m$, as required for \ref{itm:bip-extremal}.
	\end{proof}

	Finally, here is the proof of \Cref{lem:three-colours}.

	\begin{proof}[Proof of \Cref{lem:three-colours}]
		We prove the lemma by induction on $m$. When $m = 0$, the statement is trivial: take $X = Y = Z = \emptyset$ and $W = V(G)$, to see that \ref{itm:three-extremal} holds.

		Now let $m \ge 1$ and suppose that the statement holds for $m-1$. Let $G$ be a $3$-colouring of $K_{4m+1}$ with no monochromatic $\cm(m+1)$. 
		Suppose that $\{A, B\}$ is a partition of $V(G)$ such that $|A| = 2m$, $|B| = 2m+1$, and all $[A, B]$ edges are not green. By \Cref{lem:bip}, there are partitions $\{X, Y\}$ of $A$ and $\{Z, W\}$ of $B$ satisfying \ref{itm:bip-extremal}. Without loss of generality, $|Z| \le m$. 
		It is easy to verify, using that $G$ does not have a monochromatic $\cm(m+1)$, that all $[Z, W]$ edges are green; if $|W| \ge m+2$ then all edges in $W$ are green; and if $Z \neq \emptyset$ then all $[X, Y]$ edges are green. It follows that property \ref{itm:three-extremal} holds.

		We may now assume that there is no partition $\{A, B\}$ of $V(G)$ such that the $[A, B]$ edges avoid one of the three colours. As $|G| = 4m+1$, this implies that there is a most one monochromatic component in each colour that contains a matching of size $m$, and every such component has order at least $2m+2$. By \Cref{cor:ge-remove-vx}, applied to each such component, there is a set $U$ of three vertices such that $G \setminus U$ has no monochromatic $\cm(m)$. This is a contradiction if $m = 1$ (because $|G \setminus U| \ge 2$, and so $G \setminus U$ contains a monochromatic $\cm(m)$). If $m \ge 2$, take $v \in V(G) \setminus U$, then by induction applied to $G' = G \setminus (U \cup \{v\})$, there is a partition $\{X, Y, Z, W\}$ of $V(G')$ as in \ref{itm:three-extremal} (with $m-1$). Let $w \in W$. It is easy to check that if $vw$ has colour $\ell$, there is an $\ell$-coloured $\cm(m)$ in $G \setminus U$, contradicting the choice of $U$.
	\end{proof}

	One can use similar ideas to prove a stability result for $3$-coloured complete graphs that do not contain a large monochromatic connected matching; we sketch a proof of such a result below. Such a result implies a stability result for monochromatic connected matchings in $3$-coloured almost complete graphs via \Cref{thm:matching-complete}, which in turn implies a similar stability result for paths in $3$-coloured complete graphs.

	\begin{proof}[Sketch of proof of a stability result for $3$-coloured complete graphs]
		We sketch a proof of the following statement: if $G$ is a $3$-coloured $K_{4m}$ that does not contain a monochromatic $\cm((1+\eps)m)$, then, up to relabelling of the colours, there is a partition $\{X, Y, Z, W\}$ of $V(G)$ such that $|X| = |Y| = m$; almost all $[X, Z]$ and $[Y, W]$ edges are red; and almost all $[X, W]$ and $[Y, W]$ edges are blue. Here $\eps > 0$ is small and $m$ is large.

		Given such $G$, if there are disjoint sets $A$ and $B$ of size at least $(2 - 8\eps)m$ such that there are no green $[A, B]$ edges, then we can apply similar arguments as in the proof of \Cref{lem:bip} to find the required partition $\{X, Y, Z, W\}$. Now suppose that there exist no disjoint $A$ and $B$ of size at least $(2 - 8\eps)m$ such that the $[A, B]$ edges avoid at least one colour. It follows that there is at most one component in each colour that contains a matching of size at least $(1 - 4\eps)m$. By \Cref{cor:ge-remove-vx} and the assumption that $G$ does not contain a monochromatic $\cm((1+\eps)m)$, there is a set $U$ of size at most $15\eps m$ such that $G \setminus U$ does not contain a monochromatic $\cm((1 - 4\eps)m)$. 
		So $G \setminus U$ is a graph on at least $(4 - 15\eps)m \ge 4(1 - 4\eps)m + 1$ vertices which does not have a monochromatic $\cm((1 - 4\eps)m)$, contradicting \Cref{lem:three-colours}.
	\end{proof}

\subsection*{Acknowledgements}

	I would like to thank Louis DeBiasio for his helpful suggestions regarding an earlier draft of this paper. I would also like to thank the anonymous referee for their helpful comments.

\bibliography{matchings}
\bibliographystyle{amsplain}

\end{document}